\documentclass[a4paper,11pt,leqno]{amsart}
\usepackage[T1]{fontenc}
\usepackage[utf8]{inputenc}
\usepackage[margin=1.2in]{geometry}
\usepackage{lmodern}
\usepackage[english]{babel}
\usepackage{graphicx}
\usepackage{amsmath}
\usepackage{amsfonts}
\usepackage{amssymb} 
\usepackage{amsthm}
\usepackage{bbm}
\usepackage{bm} 
\usepackage{lipsum}
\usepackage{xcolor}
\usepackage{url}
\usepackage{hyperref}
\allowdisplaybreaks
\theoremstyle{plain}
\newtheorem{thm}{Theorem}
\newtheorem{lem}[thm]{Lemma}
\newtheorem{prop}[thm]{Proposition}

\newtheorem{defn}[thm]{Definition}
\newtheorem{rem}[thm]{Remark}

\def \be {\begin{equation}}
\def \ee {\end{equation}}

\def \r {}
\def \S{\Sigma}
\def \a{\alpha}
\def \ra{\rightarrow}
\def \bs {\boldsymbol}
\def \E {\mathbb{E}}

\def \R {\mathbb{R}}

\begin{document}

\title[On the convergence in MFG: a two state model without uniqueness]{On the convergence problem in Mean Field Games: a two state model without uniqueness}
\author{Alekos Cecchin}
\address[A. Cecchin, P. Dai Pra, M. Fischer, and G. Pelino]
{\newline \indent Department of Mathematics ``Tullio Levi Civita'', 
 University of Padua \newline
\indent Via Trieste 63, 35121 Padova, Italy
\newline }
\email[]{alekos.cecchin@math.unipd.it, daipra@math.unipd.it, fischer@math.unipd.it, \newline \indent   
	 guglielmo.pelino@math.unipd.it}
\author{Paolo Dai Pra}
\author{Markus Fischer}
\author{Guglielmo Pelino}
\thanks{The first and last author are partially supported by the PhD Program in Mathematical Science, Department of Mathematics, University of Padua (Italy), Progetto Dottorati - Fondazione Cassa di Risparmio di Padova e Rovigo. 
The second, third and fourth author acknowledge financial support through the project ``Large Scale Random Structures'' of the Italian Ministry of Education, Universities and Research (PRIN 20155PAWZB-004).
The first and third author are also supported by the research projects ``Nonlinear Partial Differential Equations: Asymptotic Problems and Mean Field Games'' of the Fondazione CaRiPaRo, and ``Mean Field Games and Nonlinear PDEs'' of
the University of Padua
}
\date{\today}
\subjclass{60F99, 60J27, 60K35, 91A13, 93E20} %
\keywords{Mean field game, finite state space, jump Markov process, $N$-person games, Nash equilibrium, master equation, propagation of chaos, non-uniqueness}
\begin{abstract}
We consider $N$-player and mean field  games in continuous time over a finite horizon, where the position of each agent belongs to $\{-1,1\}$. 
 If there is uniqueness of mean field game solutions,  e.g. under monotonicity assumptions, then the master equation possesses a smooth solution which can be used to prove convergence of the value functions and of the feedback Nash equilibria of the $N$-player game, as well as a propagation of chaos property for the associated optimal trajectories. We study here an example with anti-monotonous costs, and show that the mean field game has exactly three solutions. 
 We prove that the value functions  converge to the entropy solution of the master equation, which in this case can be written as a scalar conservation law in one space dimension, and that the optimal trajectories admit a limit: they select one mean field game soution, so there is propagation of chaos. Moreover, viewing the mean field game system as the necessary conditions for optimality of a deterministic control problem, we show that the $N$-player game selects the {\r optimizer} of this problem.
\end{abstract}
\keywords{
  Mean field game, finite state space, jump Markov process, $N$-person games, Nash equilibrium, master equation, propagation of chaos, non-uniqueness}
\subjclass{  60F99, 60J27, 60K35, 91A13, 93E20}
 
\maketitle


\section{Introduction}

In this paper, we study a simple yet illustrative example concerning the convergence problem in finite horizon mean field games. Mean field games, as introduced by J.-M.~Lasry and P.-L.~Lions and, independently, by M.~Huang, R.P.~Malham{\'e} and P.E.~Caines (cf.\ \cite{lasrylions07, huangetal06}), are limit models for symmetric non-cooperative many player dynamic games as the number of players tends to infinity; see, for instance, the lecture notes \cite{cardaliaguet13} and the recent two-volume work \cite{carmonadelarue}. The notion of optimality adopted for the many player games is usually that of a Nash equilibrium. The limit relation can then be made rigorous in two opposite directions: either by showing that a solution of the limit model (the mean field game) induces a sequence of approximate Nash equilibria for the $N$-player games with approximation error tending to zero as $N\to \infty$, or by identifying the possible limit points of sequences of $N$-player Nash equilibria, again in the limit as $N\to \infty$, as solutions, in some sense, of the limit model. This latter direction constitutes the convergence problem in mean field games. 

Important for the convergence problem is the choice of admissible strategies and the resulting definition of Nash equilibrium in the many player games. For Nash equilibria defined in stochastic open-loop strategies, the convergence problem is rather well understood, see \cite{fischer17} and, especially, \cite{lacker16}, both in the context of finite horizon games with general Brownian dynamics. In \cite{lacker16}, limit points of sequences of $N$-player Nash equilibria are shown to be concentrated on weak solutions of the corresponding mean field game. This concept of solution is also used in another, more recent work by Lacker; see below.

Here, we are interested in the convergence problem for Nash equilibria in Markov feedback strategies with full state information. A first result in this direction was given by Gomes, Mohr, and Souza \cite{gomes} in the case of finite state dynamics. There, convergence of Markovian Nash equilibria to the mean field game limit is proved, but only if the time horizon is small enough. A breakthrough was achieved by Cardaliaguet, Delarue, Lasry, and Lions in \cite{cardaliaguetetal15}. In the setting of games with non-degenerate Brownian dynamics, possibly including common noise, those authors establish convergence to the mean field game limit, in the sense of convergence of value functions as well as propagation of chaos for the optimal state trajectories, for arbitrary time horizon provided the so-called master equation associated with the mean field game possesses a unique sufficiently regular solution. The master equation arises as the formal limit of the Hamilton-Jacobi-Bellman systems determining the Markov feedback Nash equilibria. It yields, if well-posed, the optimal value in the mean field game as a function of initial time, state and distribution. It thus also provides the optimal control action, again as a function of time, state, and measure variable. This allows, in particular, to compare the prelimit Nash equilibria to the solution of the limit model through coupling arguments.  

If the master equation possesses a unique regular solution, which is guaranteed under the Lasry-Lions monotonicity conditions, then the convergence analysis can be considerably refined. In this case, for games with finite state dynamics, Cecchin and Pelino \cite{cp} and, independently, Bayraktar and Cohen \cite{bayraktarcohen} obtain a central limit theorem and large deviations principle for the empirical measures associated with Markovian Nash equilibria. In \cite{delarueetala, delarueetalb}, Delarue, Lacker, and Ramanan carry out the analysis, enriched by a concentration of measure result, for Brownian dynamics without or with common noise.

Well-posedness of the master equation implies uniqueness of solutions to the mean field game, given any initial time and initial distribution. Here, we study the convergence problem in Markov feedback strategies for a simple example exhibiting non-uniqueness of solutions. The model has  dynamics in continuous time with players' states taking values in $\{-1,1\}$. Running costs only depend on the control actions, while terminal costs are anti-monotonic with respect to the state and measure variable. Such an example was first considered by Gomes, Velho, and Wolfram in \cite{gomestwoa, gomestwob}, where numerical evidence on the convergence behavior was presented; it should also be compared to Lacker's ``illuminating example'' (Subsection~3.3 in \cite{lacker16}) 
and to the example in Subsection~3.3 of \cite{bardifischer} by Bardi and Fischer, both in the diffusion setting.  In the infinite time horizon and finite state case, an example of non-uniqueness is studied in \cite{paolo}, via numerical simulations, where periodic orbits emerge as solutions to the mean field game. 

{\r{For the two-state example studied here}}, the mean field game possesses exactly three solutions, given any initial distribution, as soon as the time horizon is large enough. Consequently, there is no regular solution to the master equation, while multiple weak solutions exist. For the $N$-player game, on the other hand, there is a unique symmetric Nash equilibrium in Markov feedback strategies for each $N$, determined by the Hamilton-Jacobi-Bellman system. We show that the value functions associated with these Nash equilibria converge, as $N\to \infty$, to a particular solution of the master equation. In our case, the master equation can be written as a scalar conservation law in one variable (cf.\ Subsection~\ref{SubMaster}). The (weak) solution that is selected by the $N$-player Nash equilibria can then be characterized as the unique entropy solution of the conservation law. The entropy solution presents a discontinuity in the measure variable (at the distribution that assigns equal mass to both states). Convergence of the value functions is uniform outside any neighborhood of the discontinuity. We also prove propagation of chaos for the $N$-player state processes provided that their averaged initial distributions do not converge to the discontinuity. The proofs of convergence adapt arguments from \cite{cp} based on the fact that the entropy solution is smooth away from its discontinuity, as well as a qualitative property of the $N$-player Nash equilibria, which prevents crossing of the discontinuity. The entropy solution property is actually not used in the proof. In Subsection~\ref{SubPotential}, we give an alternative characterization of the solution selected by the Nash equilibria in terms of a variational problem based on the potential game structure of our example. Potential mean field games have been studied in several works in the continuous state setting, starting from \cite{card} by Cardaliaguet, Graber, Porretta and Tonon.

Let us mention three recent preprints that are related to our paper. In \cite{nutzetalii}, Nutz, San Martin, and Tan address the convergence problem for a class of mean field games of optimal stopping. The limit model there possesses multiple solutions, which are grouped into three classes according to a qualitative criterion characterizing the proportion of players that have stopped at any given time. Solutions in one of the three classes will always arise as limit points of $N$-player Nash equilibria, solutions in the second class may be selected in the limit, while solutions in the third class cannot be reached through $N$-player Nash equilibria. In \cite{lacker}, Lacker attacks the convergence problem in Markov feedback strategies by probabilistic methods. For a class of games with non-degenerate Brownian dynamics that may exhibit non-uniqueness, the author shows that all limit points of the $N$-player feedback Nash equilibria are concentrated, as in the open-loop case, on weak solutions of the mean field game. These solutions are more general than randomizations of ordinary (``strong'') solutions of the mean field game; their flows of measures, in particular, are allowed to be stochastic containing additional randomness. Still, uniqueness in ordinary solutions implies uniqueness in weak solutions, which permits to partially recover the results in \cite{cardaliaguetetal15}. The question of which weak solutions can appear as limits of feedback Nash equilibria in a situation of non-uniqueness seems to be mainly open. In \cite{delaruetchuendom}, Delarue and Foguen Tchuendom study a class of linear-quadratic mean field games with multiple solutions in the diffusion setting. They prove that by adding a common noise to the limit dynamics uniqueness of solutions is re-established. As a converse to this regularization by noise result, they identify the mean field game solutions that are selected when the common noise tends to zero as those induced by the (unique weak) entropy solution of the master equation of the original problem. The interpretation of the master equation as a scalar conservation law works in their case thanks to a one-dimensional parametrization of an a priori infinite dimensional problem. Limit points of $N$-player Nash equlibria are also considered in \cite{delaruetchuendom}, but in stochastic open-loop strategies. Again, the mean field game solutions that are selected are those induced by the entropy solution of the master equation. Interestingly, these solutions are not minimal cost solutions; indeed, the solution which minimizes the cost of the representative player in the mean field game is shown to be different from the ones selected by the limit of the Nash equilibria. In \cite{delaruetchuendom}, the $N$-player limit and the vanishing common noise limit both select two solutions of the original mean field game with equal probability. This is due to the fact that in \cite{delaruetchuendom} the initial distribution for the state trajectories is chosen to sit at the discontinuity of the unique entropy solution of the master equation. In our case, we expect to see the same behavior if we started at the discontinuity, see Section~\ref{Conclusions} below.   

{\r It is worth mentioning that the opposite situation, with respect to the one treated here, is considered in the examples presented in \cite{doncel} and in Section 7.2.5 of \cite{carmonadelarue}, Volume I. In these examples, uniqueness of mean field game solutions holds, but there are multiple feedback Nash equilibria for the $N$-player game.  This is due to the fact that in both cases the authors consider a finite action set (while for us it is continuous), so that in particular the Nash system is not well-posed. They prove that there is a sequence of (feedback) Nash equilibria which converges to the mean field game limit, but also a sequence that does not converge.}

The rest of this paper is organized as follows. In Section~\ref{SectGames}, for a class of mean field and $N$-player games with finite state space, we give the definition of $N$-player Nash equilibrium and solution of the mean field game, and introduce the corresponding differential equations, namely the $N$-player Hamilton-Jacobi-Bellman system, the mean field game system as well as the associated master equation. Section~\ref{SectExample} presents the two-state example, starting from the limit model, analyzed first in terms of the mean field game system (Subsection~\ref{SubMFG}), then in terms of its master equation (Subsection~\ref{SubMaster}). In Subsections \ref{SubValue} and \ref{SubChaos} we show that the $N$-player Nash equilibria converge to the unique entropy solution of the master equation; cf.\ Theorems \ref{value} and \ref{chaos} below for convergence of value functions and propagation of chaos, respectively. The qualitative property of the Nash equilibria used in the proofs of convergence is in Subsection~\ref{SubPrelimit}. Subsection~\ref{SubPotential} gives the variational characterization of the solution that is selected by the Nash equilibria. Concluding remarks are in Section~\ref{Conclusions}.

\section{Mean field games with finite state space} \label{SectGames}

\subsection{The $N$-player game}

We consider the continuous time evolution of the states $X_i(t)$, $i=1,2,\ldots,N$, of $N$ players; the state of each player belongs to a given finite set $\S$. Players are allowed to control, via an arbitrary {\em feedback}, their jump rates. For $i=1,2,\ldots,N$ and $y \in \S$, we denote by $\a_y^i: [0,T] \times  \S^{N} \ra [0,+\infty)$ the rate at which player $i$ jumps to the state $y \in \S$: it is allowed to depend on the time $t \in [0,T]$, and on the state ${\bm x} = (x_i)_{i=1}^N$ of all players. Denoting by $A$ the set of functions $ [0,T] \times  \S^{N} \ra [0,+\infty)$ which are measurable and locally integrable in time, we assume $\a_y^i \in A$. So we write $\a^i \in \mathcal{A} := A^{\S}$, and let ${\bs{\a}}^N \in \mathcal{A}^N$ denote the controls of all players, and will be also called {\em strategy vector}.
In more rigorous terms, for ${\bs \a}^N \in \mathcal{A}^N$, the state evolution $\bm{X}_t := (X_i(t))_{i=1}^N$ is a Markov process, whose law is uniquely determined as solution to the martingale problem for the time-dependent generator
\[
\mathcal{L}_t f({\bm x}) = \sum_{i=1}^N \sum_{y \in \S} \a^i_y (t,{\bm x}) \left[ f([{\bm x}^i,y]) - f({\bm x})\right],
\]
where
\[
[{\bm x}^i,y]_j = \left\{ \begin{array}{ll} x_j & \mbox{for } j \neq i \\ y &  \mbox{for } j = i. \end{array} \right.
\]
Now let
\[
P(\S) := \{ m \in [0,1]^{\S}: \sum_{x \in \S} m_x = 1\}
\]
be the simplex of probability measures on $\S$. To every ${\bm x} \in \S^N$ we associate 
the element of $P(\S)$
\be \label{empmeas}
m^{N,i}_{\bm{x}} :=  \frac{1}{N-1}\sum_{j=1, j \neq i}^N \delta_{x_j}.
\ee
Thus, $m^{N,i}_{\bm{X}}(t):= m^{N,i}_{\bm{X}_t} $ is the empirical measure of all the players except the $i$-th.
Given the functions
\[
L: \S \times [0,+\infty)^{\S} \ra \R, \ \ F: \S \times P(\S) \ra \R, \ \ G: \S \times P(\S) \ra \R,
\]
the feedback controls ${\bs{\a}}^N \in \mathcal{A}^N$ and the corresponding process ${\bm X}(\cdot)$, the {\em cost} associated to the $i$-th player is given by
\[
J_i^N(\boldsymbol{\alpha}^N) := \mathbb{E}\left[\int_{0}^T \left[L(X_i(t), \alpha^i(t,\bm{X}_t)) + F\left(X_i(t), m^{N,i}_{\bm{X}}(t) \right)\right]dt + G\left(X_i(T), m^{N,i}_{\bm{X}}(T) \right) \right].
\]
For a strategy vector $\boldsymbol{\alpha}^N = (\alpha^1, \dots, \alpha^N)\in\mathcal{A}^N$ and   $\beta\in\mathcal{A}$, denote by $[\boldsymbol{\alpha}^{N,-i}; \beta]$ the perturbed strategy vector given by
\begin{equation*}
[\boldsymbol{\alpha}^{N,-i}; \beta]_j := 
\begin{cases}
\alpha_j, \ \ j \neq i\\
\beta, \ \ j = i.
\end{cases}
\end{equation*}
\begin{defn} \label{def:Nash}
A strategy vector $\boldsymbol{\alpha}^N$ is a Nash equilibrium for the $N$-player game if for each $i=1, \dots, N$
\begin{equation*}
J_i^N(\boldsymbol{\alpha}^N) = \inf_{\beta\in\mathcal{A}}J_i^N([\boldsymbol{\alpha}^{N,-i}; \beta]).
\end{equation*}
\end{defn}
The search for a Nash equilibrium is based on the Hamilton-Jacobi equations that we now briefly illustrate. Define the Hamiltonian $H: \S \times \R^{\S} \ra \R$ as the Legendre transform of $L$:
\begin{equation}
\label{h}
H(x, p) := \sup_{a \in [0,+\infty)^{\S}} \left\{- (a \cdot p)_x - L(x,a)\right\},
\end{equation}
with ${\displaystyle{(a \cdot p)_x := \sum_{y \neq x} a_y p_y}}$.
We will assume the supremum in \eqref{h} is attained at an unique maximizer $a^*(x,p)$.

Given a function $ g: \Sigma \to \mathbb{R}$, we denote its first finite difference $\Delta g (x) \in \mathbb{R}^{\Sigma}$  by
\begin{equation*}
\Delta g(x):= \left(g(y) - g(x)\right)_{y \in \Sigma}.
\end{equation*}
When we have a function $g : \Sigma^N \to \mathbb{R}$, we denote with $\Delta^j g(\bm{x}) \in \mathbb{R}^{\Sigma}$ the first finite difference with respect to the $j$-th coordinate.
The Hamilton-Jacobi-Bellman system associated to the above differential game is given by:
\begin{equation} \label{hjb}
\begin{cases}
-\frac{\partial {v}}{\partial t}^{N,i}(t,\bm{x}) - \sum_{j=1, \ j\neq i}^N a^*(x_j, \Delta^j v^{N,j}) \cdot \Delta^j v^{N,i} + H(x_i, \Delta^i v^{N,i}) = F\left(x_i, m^{N,i}_{\bm{x}} \right),\\
v^{N,i}(T, \bm{x}) = G\left(x_i, m^{N,i}_{\bm{x}} \right).
\end{cases}
\end{equation}
This is a system of $N |\Sigma|^N$ coupled ODE's, indexed by $i\in\{1,\dots,N\}$ and $\bm{x}\in\Sigma^N$, whose well-posedness  for all $T > 0$ can be proved through standard ODEs techniques under regularity assumptions which guarantee that $a^*$ and $H$ are uniformly Lipschitz in their second variable. Under these conditions, the $N$-player game has a unique Nash equilibrium given by the feedback strategy ${\bs{\a}}^N \in \mathcal{A}^N$ defined by
\[
\alpha^{i,N}(t,\bm{x}) := a^{*}(x_i, \Delta^i v^{N,i}(t, \bm{x})) \qquad i=1,\ldots,N.
\]

\subsection{The macroscopic limit: the mean field game and the master equation}

The limit as $N \ra +\infty$ of the $N$-player game admits two alternative descriptions, that we illustrate here at heuristic level. Assuming the empirical measure of the process corresponding to the Nash equilibrium obeys a Law of Large Numbers, i.e. it converges to a deterministic flow in $P(\S)$, a {\em representative player} in the limit as $N\ra +\infty$ faces the following problem: 
\begin{itemize}
\item[(i)]
the player controls its jump intensities $\a_y: [0,T] \times \S \ra [0,+\infty)$, $y \in \S$, via feedback controls depending on time and on his/her own state;
\item[(ii)]
For a given deterministic flow of probability measures 
$m : [0,T] \to P(\Sigma)$, the player aims at minimizing the cost
\be \label{mfcost}
J(\alpha,m):= \mathbb{E} \left[ \int_0^T
	\left[L(X(t), \alpha(t,X(t))) + F(X(t),m(t))\right]dt +G(X(T), m(T))\right].
\ee 
\item[(iii)]
Denote by $\a^{*,m}$ the optimal control for the above problem, and let $(X^{*,m}(t))_{t \in [0,T]}$ be the corresponding optimal process. The above-mentioned Law of Large Number predicts that the flow $(m(t))_{t \in [0,T]}$ should be chosen so that the following consistency relation holds:
\[
m(t) = \text{Law}(X^{*,m}(t))
\]
foe every $t \in [0,T]$.
\end{itemize}
This is implemented by coupling the HJB equation for the control problem with cost \eqref{mfcost} with the forward Kolmogorov equation for the evolution of the $ \text{Law}(X^{*,m}(t))$, obtaining the so-called {\em Mean Field Game System}:
\begin{equation}
\tag{MFG}
\begin{cases}
-\frac{d}{dt} u(t,x) + H(x, \Delta^x u(t,x)) =  F(x, m(t)), \\
\frac{d}{dt} m_x(t) = \sum_y m_y(t) a^{*}_x(y, \Delta^y u(t,y)),\\
u(T,x) = G(x, m(T)), \\
m_x(0) = m_{x,0},
\end{cases}
\label{eqn:MFG}
\end{equation}
It is known, and largely exemplified in this paper, that well-posedness of \eqref{hjb} does not imply uniqueness of solution to \eqref{eqn:MFG}.

An alternative description of the macroscopic limit {\r stems from} the ansatz that the solution of the Hamilton-Jacobi-Bellman system \eqref{hjb} is of the form
\[
v^{N,i}(t, \bm{x}) = U^N(t, x_i, m^{N,i}_{\bm{x}}),
\]
for some $U^N: [0,T] \times \S \times P(\S) \ra \R$. Assuming $U^N$ admits a limit $U$ as $N \ra +\infty$, we formally obtain that $U$ solves the following equation, that will be referred to as the {\em master equation}:
\be \label{master} \tag{MAS}
\begin{cases}
 -\frac{\partial U}{\partial t}&\!\!\!\!\!\! \!\!\! \!\!\! \!\!\! \!\!\! \!\!\! \!(t,x,m) \!\! +\!\! H(x,\! \Delta^x U(t,x,m)) \!-\!\! \int_{\Sigma} \! D^m \! U(t,x,m,y) \! \cdot \! a^*\!(y, \Delta^y U(t,y,m)) dm(y) \!\!= \!\!F(x,m)  \\
  U(T,x,m) &\!\!\!  = G(x,m), \ \ \ (x,m) \in \Sigma \times P(\Sigma),
\end{cases}
\ee
where the derivative $D^m U : [0,T] \times \Sigma \times P(\Sigma) \times \Sigma \to \mathbb{R}^{\S}$ with respect to $m \in P(\S)$ is defined by
\begin{equation}
\label{ex:def_der}
[D^m U(t,x,m,y)]_z := \lim_{s \downarrow 0} \frac{U(t,x,m + s(\delta_z - \delta_y)) - U(t,x,m)}{s}.
\end{equation}
We conclude this section by recalling that uniqueness in both \eqref{eqn:MFG} and \eqref{master} is guaranteed if the cost function $F$ and $G$ are {\em monotone} in the Lasry-Lions sense, i.e. for every $m, m' \in P(\Sigma)$, 
 \begin{equation}
 \label{mon}
 \sum_{x \in \Sigma} (F(x,m) - F(x,m')){\r(m_x - m'_x)} \geq 0, 
 \end{equation}
and the same for $G$. We are interested here in examples that violate this monotonicity condition.

\section{An example of non uniqueness} \label{SectExample}

We consider now a special example within the class of models described above. We let $\S := \{-1,1\}$ be the state space. An element $m \in P(\S)$ can be identified with its mean {\r $m_1 - m_{-1}$}; so from now we {\r write $m \in [-1,1]$ to denote the mean, while the element of $P(\S)$ will be denoted only in vector form $(m_1, m_{-1})$.}  We also write $\a^i(t,{\bm x})$ for $\a^i_{-x_i}(t,{\bm x})$, i.e. the rate at which player $i$ {\em flips} its state from $x_i$ to $-x_i$.
Moreover we choose
\[
L(x,a) := \frac{a^2}{2}, \qquad F(x,m) \equiv 0, \qquad G(x,m) := -mx.
\]
The final cost favors alignment with the majority, while the running cost is a simple quadratic cost. 
Compared to condition \eqref{mon}, note that the final cost is {\em anti-monotonic}, as 
\[
 \sum_{x \in \Sigma} (G(x,m) - G(x,m')){\r(m_x - m'_x)} = -(m-m')^2 \leq0.
 \]
The associated Hamiltonian is given by
\begin{equation}
\label{ex:min}
H(x,p) = {\r \sup_{a \geq 0} \left\{a p_{-x} - \frac{a^2}{2} \right\} = \frac{(p^-_{-x})^2}{2},}
\end{equation}
with $a^*(x,p) = p_{-x}	^-$, where $p^-$ denotes the negative part of $p$. From now on, we identify $p$ with $p_{-x}\in\mathbb{R}$ and $\Delta^x u$ with its non-zero component $u(-x)-u(x)$.

\subsection{The mean field game system} \label{SubMFG}

The first equation in \eqref{eqn:MFG}, i.e the HJB equation for the value function $u(t,x)$, reads,
using \eqref{ex:min},

\begin{equation}
\label{eqn:hjb0}
\begin{cases}
-\frac{d}{dt} u(t,x) + 
\frac12
\left[(\Delta^x u(t,x))^-\right]^2 =  0 \\ 
u(T,x) = -m(T)x
\end{cases}
\end{equation}
Now define $z(t) := u(t,-1) - u(t,1)$. Subtracting the equations \eqref{eqn:hjb0} for $x = \pm 1$ and observing that
\[
\left[(\Delta^x u(t,-1))^-\right]^2 - \left[(\Delta^x u(t,1))^-\right]^2 = z|z|,
\]
we have that $z(t)$ solves
\be
\begin{cases}
\dot{z} = \frac{z|z|}{2}\\ 
z(T)= 2m(T).
\end{cases}
\label{ex:hjb}
\ee
This equation must be coupled with the forward Kolmogorov equation, 
i.e. the second equation in \eqref{eqn:MFG},
that reads $\dot{m} = -m|z| +z$. The mean field game system takes therefore the form:
\be
\begin{cases}
\dot{z} = \frac{z|z|}{2}\\
\dot{m} = -m|z| +z\\
z(T)=2m(T)\\
m(0)=m_0.
\end{cases}
\label{mfg}
\ee

\begin{prop}
\label{prop1}
Let $T(m_0)$ be the unique solution in $\left[\frac12,2\right]$ of the equation
\be 
\label{smalltime}
|m_0| = \frac{(2T-1)^2 (T+4)}{27 T}.
\ee
Then, for every $m_0 \in [-1,1] \setminus \{ 0 \}$, system \eqref{mfg} admits
\begin{itemize}
\item[(i)]
a unique solution for $T<T(m_0)$;
\item[(ii)]
two distinct solutions for $T = T(m_0)$;
\item[(iii)] 
three distinct solutions for $T > T(m_0)$.
\end{itemize}
If $m_{0} = 0$, then $T(0) = 1/2$ and \eqref{mfg} admits
\begin{itemize}
\item[(i)]
a unique solution for $T \leq 1/2$;
\item[(ii)] 
three distinct solutions for $T > 1/2$:  the constant zero solution, $(z_{+},m_{+})$, and $(z_{-},m_{-})$, where $m_{+}(t) = -m_{-}(t) > 0 $ for every $t\in (0,T]$.
\end{itemize}
\end{prop}

\begin{proof}
Note that \eqref{ex:hjb} can be solved as a final value problem, giving
\be
\label{solhjb}
z(t) =  \frac{2m(T)}{|m(T)|(T-t) +1}.
\ee
This can then be inserted in the forward Kolmogorov equation
$\dot{m} = -m|z| +z$,
giving as unique solution
\be
\label{solkol}
m(t) = (m_0 - \text{sgn}(m(T)))\left(\frac{|m(T)|(T-t)+1}{|m(T)|T+1}\right)^2 + \text{sgn}(m(T)).
\ee
These are actually solutions of \eqref{mfg} if and only if the consistency relation obtained by setting $t=T$ in \eqref{solkol} holds, i.e.  if and only if $m(T) = M$ solves
\be
\label{cons}
T^2 M^3 + T(2-T)M |M| + (1-2T)M - m_0 =0.
\ee
Moreover, distinct solutions of \eqref{cons} correspond to distinct solutions of \eqref{mfg}. 
We first look for nonnegative solutions of \eqref{cons}. Set
\[
f(M) := T^2 M^3 + T(2-T)M^2  + (1-2T)M - m_0.
\]
Note that 
\[
f'(M) < 0 \ \ \iff \ \ M \in \left(-\frac{1}{T}, \frac{2T-1}{3T} \right).
\]
If $T \leq \frac12$ then $f$ is strictly increasing in $(0,+\infty)$, so the equation $f(M) = 0$ admits a unique nonnegative solution if $m_0 \geq 0$, otherwise there is no nonnegative solution. If $T > \frac12$, then $f$ restricted to $(0,+\infty)$ has a global minimum at $M^* = \frac{2T-1}{3T}$. If $m_0>0$ then there is still a unique nonnegative solution, while for $m_0 = 0$ there are two nonnegative solution, one of which is zero. If, instead, $m_0 < 0$, so that $f(0) > 0$, the equation $f(M) = 0$ has zero, one or two nonnegative solutions, depending on whether $f(M^*) >0$, $f(M^*) = 0$ or $f(M^*) < 0$ respectively. Observing that
\[
f(M^*) = -m_0 - \frac{(2T-1)^2 (T+4)}{27 T},
\]
we see that those three alternatives occur if $T < T(m_0)$, $T = T(m_0)$ and $T > T(m_0)$ respectively. The case $M\leq 0$ is treated similarly.  

\end{proof}

\subsection{The Master Equation} \label{SubMaster}

Identifying again a probability on $\S$ with its mean $m$, {\r using the expression for $H$ and its minimizer given in \eqref{ex:min}}, the Master Equation \eqref{master} takes the form
\be
\label{ex:master}
\begin{cases}
 -\frac{\partial U}{\partial t}(t,x,m)&\!\!\!\!\!  +  \frac12 \left[ \left(\Delta^x U(t,x,m)\right)^- \right]^2 \! -{\r D^m U(t,x,m,1)} \left(\Delta^x U(t,1,m)\right)^- \frac{1+m}{2}  \\
&\hspace{1.2 cm}-{\r D^m U(t,x,m,-1)} \left(\Delta^x U(t,-1,m)\right)^- \frac{1-m}{2}= F(x,m),  \\
U&\!\!\!\!\!\!\!\!\!\!\!\!\!\!\!\!\!\!\!\!\!\!\!\!\!\!\!\!\!\!\!\!\!\!\!(T,x,m) = G(x,m), \ \ \ (x,m) \in \{-1,1\} \times [-1,1].
\end{cases}
\ee
{\r In \eqref{ex:master}, the derivative $D^m U$ is still intended in the sense introduced in \eqref{ex:def_der}, but identifying the resulting vector with its non-zero component (e.g. $D^m U(t,x,m,1) = [D^m U(t,x,m,1)]_{-1}$ $=  \frac{\partial}{\partial(m_{-1} - m_1)}U(t,x,m)$). Similarly, we identify the vector $\Delta^x U$ with its non-zero component.}
Setting 
\[
{\r Z(t,m) := U(T- t, -1,m) - U(T- t,1,m),}
\]
we easily derive a closed equation for $Z$:
\be
\label{conslaw}
\begin{cases}
\frac{\partial Z}{\partial t} +  
 \frac{\partial}{\partial m} \left(m \frac{Z|Z|}{2} - \frac{Z^2}{2}\right) =0, \\
Z(0,m) =2m, 
\end{cases}
\ee 
{\r where $\frac{\partial}{\partial m}$ is denoting the differentiation in the usual sense with respect to $m \in [-1,1]$. In particular, observe that $\frac{\partial}{\partial m} = \frac{1}{2}\frac{\partial}{\partial (m_{-1} - m_1)} $.}

Note that this equation has the form of a scalar {\em conservation law}
\be \label{gencons}
\begin{cases}
\frac{\partial Z}{\partial t}(t,m) +  \frac{\partial}{\partial m} \mathfrak{g}(m,Z(t,m)) =0 \\
Z(0,m) = \mathfrak{f}(m).
\end{cases}
\ee
 Scalar conservation laws typically possess unique smooth solutions for small time, but develop singularities in finite time: weak solutions exist but uniqueness may fail. To recover uniqueness the notion of {\em entropy solution} is introduced. A simple sufficient condition can be given for piecewise smooth functions (see \cite{dafermos}).

\begin{prop} \label{lax}
Let $Z(t,m)$ be a piecewise $\mathcal{C}^1$ function, which is $\mathcal{C}^1$ outside a $\mathcal{C}^1$ curve $m = \gamma(t)$, and assume the following conditions hold:
\begin{itemize}
\item[(i)]
$Z$ solves \eqref{gencons} in the classical sense outside the curve $m = \gamma(t)$.
\item[(ii)]
The initial condition $Z(0,m) = \mathfrak{f}(m)$ holds for every $m$;
\item[(iii)]
 Denoting 
\[
Z_+(t) := \lim_{m \downarrow \gamma(t)} Z(t,m) , \ \ \ Z_-(t) := \lim_{m \uparrow \gamma(t)} Z(t,m),
\]
we have that, for every  $t \geq 0$ and every $c$ strictly between $Z_-(t)$ and $Z_+(t)$,
\be \label{rh}
\dot{\gamma}(t) = \frac{\mathfrak{g}(\gamma(t),Z_-(t)) - \mathfrak{g}(\gamma(t), Z_+(t))}{Z_-(t)-Z_+(t)},
\ee
\be \label{eqlax}
\frac{\mathfrak{g}(\gamma(t),c) - \mathfrak{g}(\gamma(t), Z_+(t))}{c-Z_+(t)} < \dot{\gamma}(t) < \frac{\mathfrak{g}(\gamma(t),c) - \mathfrak{g}(\gamma(t), Z_-(t))}{c-Z_-(t)}.
\ee
\end{itemize}
Then, $Z$ is the unique entropy solution of \eqref{gencons}.
\end{prop}

Condition \eqref{rh} is called the {\em Rankine-Hugoniot condition}, while  \eqref{eqlax} is called the {\em Lax condition}. When specialized to the case $\mathfrak{g}(m,z) := m \frac{z|z|}{2} - \frac{z^2}{2}$ and $\gamma(t) \equiv 0$ we simply obtain
\be \label{entr}
Z_+(t) = - Z_-(t) \geq 0.
\ee
For equation \eqref{conslaw}, the entropy solution can be explicitly found. Let
\be
g(M,t,m) := t^2 M^3 + t(2-t)M |M| + (1-2t)M - m
\ee
and $M(t,m)$ denote the unique solution to $g(M,t,m) =0$ with the same sign of $m$, if $m\neq 0$; 
$M$ is defined for any time and let $M(t,0) \equiv 0$. Define
\be
Z(t,m) := \frac{2 M(t,m)}{t|M(t,m)| +1}:
\label{zstar}
\ee
such function has a unique discontinuity in $m=0$, for {\r $t>1/2$}, and is 
$\mathcal{C}^1$ outside. {\r However, observe that equation \eqref{conslaw} must be solved in the finite interval $t \in  [0,T]$, where $T$ is the final time appearing in \eqref{ex:master}. Thus, for $T < 1/2$ the solution is regular.}
\begin{thm}
The function $Z$ defined in \eqref{zstar} is the unique entropy admissible weak solution to \eqref{conslaw}.
\end{thm}
\begin{proof}
From the properties of $g(M,t,m)$, it follows that
$$
\lim_{m \downarrow 0} M(t,m)= - \lim_{m \uparrow 0} M(t,m) \geq 0,
$$ 
for any time. These limits correspond to the solutions $m_+$ and $m_-$ of Proposition~\ref{prop1}, evaluated at the terminal time. Therefore \eqref{entr} is satisfied. We remark that the conservation law is set in the domain $[-1,1]$ without any boundary condition, but this is not a problem as we have invariance of the domain under the action of the characteristics. 
\end{proof}

\begin{rem}
We observe that to the entropy solution \eqref{zstar} of \eqref{conslaw} there corresponds a unique solution of \eqref{ex:master}. {\r It can be constructed via the method of characteristic curves, in terms of a specific solution to the mean field game system for the couple $(u,m)$, the one that corresponds to the solution to \eqref{mfg} employed in the definition of \eqref{zstar}.}
\end{rem}

It is known that, if there were a regular solution to the  master equation \eqref{conslaw},  thus Lipschitz in $m$, then this  solution would provide a unique solution to the mean field game system \eqref{mfg}, since the KFP equation would be well posed for any initial condition, when using  $z(t)=Z(T-t,m(t))$ induced by the solution to the master equation:
\be
\begin{cases}
&\dot{m}=-m |Z(T-t,m)| + Z(T-t,m)\\
&m(0)=m_0.
\end{cases}
\label{induced}
\ee
In our example there are no regular solutions to the master equation; however the entropy solution still induces a unique mean field game solution, if $m_0 \neq 0$.

\begin{prop}
Let $Z$ be the entropy solution defined in \eqref{zstar}. Then \eqref{induced} admits a unique solution $m^*$, for any $T$, if $m_0\neq 0$: it is the unique solution which does not change sign, for any time.
\label{mindotta}
\end{prop}

\begin{proof}
Let $m_0>0$. If $t$ and $|m-m_0|$ are small  then 
 $Z(T-t,m)$ is regular (Lipschitz-continuous) and remains positive.
  So we have a unique solution to \eqref{induced}, for small time $t\in[0,t_0]$; moreover 
   it is such that $\dot{m}>0$ and hence in particular $m(t_0)>m_0$. Thus we can iterate this procedure starting from $m(t_0)>0$: we end up with the required solution, which is positive and such that $m(t)>m_0$ for any time. This solution is unique (for any $T$) since $Z(t,m)$ is Lipschitz for $m \in [m_0,1]$. In fact the other two solutions described in Proposition \ref{prop1} would require  the vector field $Z$ in \eqref{induced} to be negative for any time, and this is not possible when considering the entropy solution $Z$. The same argument gives the claim when $m_0<0$. 
\end{proof}

\subsection{Properties of the $N+1$-player game} \label{SubPrelimit}
 
We consider now the game played by $N+1$ players, labeled by the integers $\{0,1,\ldots,N\}$. By symmetry, we can interpret the player with label $0$ as the {\em representative player}. Let
\[
\mu^{N}_{\bm x} := \frac{1}{N} \sum_{i=1}^N \delta_{x_i = 1} \in \left\{0,\frac{1}{N}, \frac{2}{N}, \ldots, \frac{N-1}{N},1 \right\}
\]
be the fraction of the ``other'' players having state $1$. Comparing with the notations in \eqref{empmeas}, note that $\mu^{N}_{\bm x} = \frac{1+m_{\bm x}^{N+1,0}}{2}$. 
In what follows, we use $N$ rather than $N+1$ as apex in all objects related to the $N+1$-player game.
By symmetry again, the value function $v^{N,0}(t, {\bm x})$ introduced in \eqref{hjb} is of the form
\[
v^{N,0}(t, {\bm x}) = V^N(t,x_0,\mu^{N}_{\bm x}),
\]
where $V^N: [0,T] \times \{-1,1\} \times \left\{0,\frac{1}{N}, \frac{2}{N}, \ldots, \frac{N-1}{N},1 \right\} \ra \R$. Since the model we are considering, besides permutation invariance, is invariant by the sign change of the state vector, it follows that
\begin{equation}
\label{eqn:sym}
V^N(t,1,\mu^{N}_{\bm x}) = V^N(t,-1,1-\mu^{N}_{\bm x}).
\end{equation}
We can therefore redefine $V^N(t,\mu) := V^N(t,1,\mu)$; from the HJB systems \eqref{hjb} we derive the following closed equation for $V^N$:
\begin{equation} 
\label{V^N}
\begin{cases}
- \frac{d}{dt} V^N(t,\mu) & \!\!\!\!\!\!+ H(V^N\!({\r t,}1-\mu) \!-\! V^N({\r t,}\mu)) 
 \!=\! N \mu \! \left[V^N\!({\r t,}1-\mu) \!-\! V^N\!({\r t,}\mu)\right]^{\!-} \!\!\left[V^N\!\!\left({\r t,}\mu \!-\! \frac{1}{N}\right) \!-\! V^N\!({\r t,}\mu)\right] \\
& + N(1-\mu)\left[V^N\left({\r t,}\mu + \frac{1}{N}\right) - V^N\left({\r t,}1-\mu-\frac{1}{N}\right)\right]^{\!-} \!\!\left[V^N\left({\r t,}\mu + \frac{1}{N}\right) - V^N({\r t,}\mu)\right]\\
V^N(T,\mu) =&\!\!\!\!\!\! -(2\mu - 1),
\end{cases}
\end{equation}
with $H(p) = \frac{(p^-)^2}{2}$. {\r It is easy to check that, when imposing a final datum $V^N(T,\mu) \in [-1,1]$, any solution to system \eqref{V^N} is such that $V^N(t,\mu) \in [-1,1]$ for any $t < T$. The locally Lipschitz property of the vector field is thus enough to conclude the existence and uniqueness of solution for any $T >0$ for the above system with $|V^N(t,\mu)| \leq 1$.
Such solution allows }to obtain the unique Nash equilibrium, given by the feedback strategy
\be
\alpha^{0,N}(t, {\bm x}) = \left\{ \begin{array}{ll} \left[V^N(t, 1-\mu^N_{\bm x}) - V^N(t,\mu^N_{\bm x}) \right]^- & \mbox{for } x_0 = 1 \\  \left[V^N(t, 1-\mu^N_{\bm x}) - V^N(t,\mu^N_{\bm x}) \right]^+ & \mbox{for } x_0 =-1. \end{array} \right.
\label{optimalnash}
\ee
We now set 
\[
Z^N(t,\mu) := V^N(t,1-\mu) - V^N(t,\mu).
\]
The following result, that will be useful later, shows that if the representative player agrees with the majority, i.e. $x_0 = 1$ and $\mu^N_{\bm x} \geq \frac12$, or $x_0 = -1$ and $\mu^N_{\bm x} \leq \frac12$, then she/he keeps her/his state by applying the control zero.  

\begin{thm}
{\r For any $\mu\in S_N =\left\{0,\frac1N, \dots, 1\right\}$,} we have
\begin{align}
	&Z^N(t,\mu) \geq 0 \quad (\alpha^N(t,1,\mu) =0)  \quad \mbox{ if } \mu \geq \frac12  \label{Z1},\\
	&Z^N(t,\mu) \leq 0 \quad (\alpha^N(t,-1,\mu) =0) 	\quad \mbox{ if } \mu \leq \frac12. \label{Z2}
\end{align}
\label{charac}
\end{thm}

\begin{proof}
We prove \eqref{Z1}, the proof of \eqref{Z2} is similar.  {\r For any $N$ even}, observe that $Z^N(\frac12) = 0$, so that it is enough to prove the claim for $\mu \geq \frac12 + \frac1N$. Define
\[
W^N(t,\mu) := V^N({\r t,} \mu)-V^N(t,\mu+\frac1N).
\]
By \eqref{V^N},
\begin{equation} \label{Z^N}
\begin{split}
\frac{d}{dt} Z^N(t,&\mu)   \!=\! H(-Z^N(t,\mu)) \!-\! H(Z^N(t,\mu))     \\ & \!+N\mu \!\left\{\left(Z^N(t,\mu)\right)^- \!W^N\!\left({\r t,}\mu-\frac1N \right) \left(Z^N\left(t,\mu-\frac1N \right)\right)^- W^N\left({\r t,}1-\mu\right)\right\} \\
& - N(1-\mu) \left\{\left(Z^N\left(t,\mu+\frac1N\right)\right)^+W^N(t,\mu) {\r +} \right. \left. \left(Z^N\left(t,\mu\right)\right)^+W^N\left(t,1-\mu-\frac1N\right) \right\}
\end{split}
\end{equation}
and
\begin{equation} \label{W^N}
\begin{split}
\frac{d}{dt}W^N(t,\mu) &=  H(Z^N(t,\mu)) -  H\left(Z^N\left(t,\mu+ \frac1N \right)\right) \\
& -N\mu  \left(Z^N(t,\mu)\right)^- W^N\left(t, \mu-\frac1N \right)   \\
& + N\left(\mu + \frac1N\right)\left(Z^N\left(t,\mu + \frac1N \right) \right)^- W^N(t,\mu) \\
& + N(1-\mu) {\r \left(Z^N\left({\r t,}\mu + \frac1N\right)\right)^+} W^N(t,\mu) \\
& -N\left(1-\mu-\frac1N\right)  {\r \left(Z^N\left({\r t,}\mu + \frac2N\right)\right)^+ }W^N\left(t,\mu+\frac1N\right).
\end{split}
\end{equation}
Note that, for $\mu > \frac12$, $Z^N(T,\mu) = 4\mu -2 >0$ and $W^N(T,\mu) = \frac2N>0$. So, set
\[
s := \sup\left\{t \leq T : Z^N(t,\nu) \leq 0 \mbox{ or } W^N(t,\nu) \leq 0 \mbox{ for some } \nu > \frac12\right\}.
\]
We complete the proof by showing that $s = -\infty$. Assume $s > -\infty$. For $t \in [s,T]$ we have $Z^N(t,{\r \mu}) \geq 0$ and $W^N(t,{\r \mu}) \geq 0$ for all ${\r \mu} > \frac12$, so, {\r from \eqref{Z^N}, observing that the terms in $\left(Z^N\right)^-$ disappear,}
\begin{equation*}
\begin{split}
\frac{d}{dt}Z^N(t,\mu) & \leq H({\r - }Z^ N(t,\mu)) +N(1-\mu)Z^N(t,\mu) W^N\left(t,1-\mu-\frac1N\right) \\
& = Z^N(t,\mu) \left[ \frac12 Z^N(t,\mu) + N(1-\mu) W^N\left(t,1-\mu-\frac1N\right) \right].
\end{split}
\end{equation*}
Since the control zero is suboptimal, it follows that {\r $|V^N(t,\mu)| \leq 1$} for all {\r $t,\mu$}, so that {\r $|Z^N(t,\mu)| \leq 2$} and {\r $|W^N(t,\mu)| \leq 2$}.  Therefore, for $t \in [s,T]$, $Z^N(t,\mu)$ is bounded from below by the solution of
\be
\label{eqz}
\begin{split}
\frac{d}{dt}z(t) & = z(t) \left[1+2N(1-\mu)\right] \\
z(T) & = 4\mu -2
\end{split}
\ee
which is strictly positive for all times. In particular $Z^N(s,\mu) >0$. Similarly, for $t \in [s,T]$, from \eqref{W^N}
\[ 
\frac{d}{dt}W^N(t,\mu) \leq  N(1-\mu) Z^{\r N}\left(\,t,\mu + \frac1N\right)W^N(t,\mu)  \leq 2N(1-\mu) W^N(t,\mu),
\]
{\r which} implies that also $W^N(s,\mu)>0$; by continuity in time, this contradicts the definition of $s$.
{\r Finally, observe that in the proof we fixed $N$ even. The proof for $N$ odd can be easily adapted with a bit of care, noting that $\mu = \frac{1}{2}$ cannot hold.} 
\end{proof}

\subsection{Convergence of the value functions} \label{SubValue}

We now consider the value function $V^N$, the unique solution  to equation \eqref{V^N}, and study its limit as $N \rightarrow +\infty$. We show that its limit corresponds to the entropy solution  of the Master Equation \eqref{ex:master}. More precisely, let $U$ be the solution to \eqref{ex:master} corresponding to the entropy solution $Z$ of \eqref{conslaw}. Define, for $\mu \in [0,1]$
\[
U^*(t,\mu) := U\left(t,1,2\mu-1\right).
\]
Note that, for $T > \frac{1}{2}$, $U^*(t, \cdot)$ is discontinuous at $\mu = \frac12$, but it is smooth elsewhere. Next result establishes that $V^N$ converges to $U^*$ uniformly outside any neighborhood of $\mu = \frac12$. In what follows, $S_N := \left\{0,\frac1N, \frac2N,\ldots,1 \right\}$.
\begin{thm}[Convergence of value functions] 
\label{value}
For any $\varepsilon>0$, $t\in [0,T]$ and $\mu\in S^N \setminus \left( \frac 12 - \varepsilon, \frac 12 + \varepsilon \right)$ we have
\be
|V^N(t,\mu) - U^*(t,\mu)|\leq \frac{C_\varepsilon}{N},
\label{14}
\ee
where {\r  $C_{\varepsilon}$} does not depend on $N$ nor on $t,\mu$, but $\lim_{\varepsilon\rightarrow0} C_\varepsilon = +\infty$.
\end{thm}

The proof of Theorem \ref{value} is based on the arguments developed in \cite{cp}. We first slightly extend the above notation, letting, for $x \in \{-1,1\}$
\[
U^*(t,x,\mu) := U(t,x,2\mu-1).
\]
Moreover, let
\[
v^{N,i}(t, \bm{x}) = V^N (t,x_i, \mu^{N,i}_{\bm{x}}), \qquad u^{N,i}(t, \bm{x}) = U^* (t,x_i, \mu^{N,i}_{\bm{x}})
\]
for $i=0,\dots,N$, where $\mu^{N,i}_{\bm{x}}= \frac{1}{N} \sum_{j=0, \\ j\neq i}^N \delta_{\{x_i = 1\}}$ is the fraction of the other players in 1. Let also $S_N^\varepsilon:= S_N \setminus (\frac12 -\varepsilon, \frac12 +\varepsilon)$.
The following results are  the adaptations of Propositions 3 and 4 of \cite{cp}. The first  provides a bound for $\Delta^j u^{N,i}(t,\bm{x})$, while the second shows that $U^*$ restricted to $S_N^\varepsilon$  is "almost" a solution of \eqref{V^N}.
\begin{prop}
\label{propcp}
For any  $t \in [0,T]$, $\varepsilon > 0$ and any $\bm{x}$ such that $\mu^{N,i}_{\bm{x}} \in S_{N}^\varepsilon$, if $ N\geq \frac2\varepsilon$, we have
\be
\Delta^j u^{N,i}(t,\bm{x}) = -\frac{1}{N}\frac{\partial}{\partial \mu} U(t,x_i,\mu^{N,i}_{\bm{x}}) + \tau^{N,i,j}(t,\bm{x}), 
\ee
for any $j \neq i$, with $\left|\tau^{N,i,j}(t,\bm{x})\right| \leq \frac{C_\varepsilon}{N^2}$. The constant $C_\varepsilon$ is proportional to the Lipschitz constant of the master equation outside the discontinuity, which behaves like $\varepsilon^{-\frac23}$.
\end{prop}
\begin{prop}
\label{prop3}
For any $t \in [0,T]$, any $\varepsilon > 0$ and any $\mu$ such that either $\mu \in [\frac{1}{2} + \varepsilon, 1]$ or $\mu \in [0, \frac{1}{2} - \varepsilon]$, the function $U^*(t,\mu)$ satisfies
\begin{align}
\label{eqn:u_N}
- \frac{d}{dt} U^*(t,\mu) &+ H(U^*({\r t,} 1-\mu) - U^*({\r t,} \mu))\\
& = N \mu \left[ U^*({\r t,} 1-\mu) -U^*({\r t,} \mu)\right]^- \left[U^*\left({\r t,} \mu - \frac{1}{N}\right) - U^*({\r t,} \mu)\right]+ r^N(t,\mu) \nonumber\\
& + N(1-\mu)\left[U^*\left({\r t,} \mu + \frac{1}{N}\right) - U^*\left({\r t,} 1-\mu-\frac{1}{N}\right)\right]^- \left[U^*\left({\r t,} \mu + \frac{1}{N}\right) - U^*({\r t,} \mu)\right] \nonumber,
\end{align} 
with $\left|r^N(t,\mu)\right| \leq \frac{C_\epsilon}{N}$, 
where $C_\varepsilon$ is as above. 
\end{prop}

We now use the information provided by Theorem \ref{charac}. Set
\[
\Sigma_N^{\varepsilon} := \left\{{\bm x}\in\Sigma^{N+1}: \sum_{i=0}^{N} \delta_{x_i=1} \not\in \left(\frac{N}{2} - N\varepsilon, \frac{N}{2} + N\varepsilon+1 \right) \right\}.
\]
If ${\bm x} \in \Sigma_N^{\varepsilon}$, then $\mu^{N,i}_{\bm{x}} \in S_N^{\varepsilon}$ for all $i$. Denote by ${\bm Y}_s$ the state at time $s$ of the $N+1$ players corresponding to the Nash equilibrium.
By Theorem \ref{charac} it follows that, if ${\bm Y}_t \in \Sigma_N^{\varepsilon}$ for some $t <T$, then  ${\bm Y}_s \in \Sigma_N^{\varepsilon}$ for all $s \in [t,T]$. In particular, by using the invariance property \eqref{eqn:sym}, we obtain
\be
v^{N,i}(s,\bm{Y}_s) \leq \max_{\mu^N \in S_N^\varepsilon} V^N(s,\mu^N),
\ee
\be
|v^{N,i}(s,\bm{Y}_s)- u^{N,i}(s,\bm{Y}_s)| \leq \max_{\mu^N \in S_N^\varepsilon} |V^N(s,\mu^N) - U^*(s,\mu^N)|,
\label{12}
\ee
for every $s\in[t,T]$, almost surely, and 
\be
\max_{\bm{x} \in \Sigma_N^\varepsilon}|v^{N,i}(s,\bm{x})- u^{N,i}(s,\bm{x})| = \max_{\mu^N \in S_N^\varepsilon} |V^N(s,\mu^N) - U^*(s,\mu^N)|.
\label{34}
\ee
Moreover, we note that
\begin{align}
|\Delta^i v^{N,i}&(s,\bm{Y}_s) - \Delta^i v^{N,i}(s,\bm{Y}_s) | \nonumber\\
&=|V^N(s,-Y_i(s), \mu^{N,i}_{\bm{Y}}(s)) - U(s,-Y_i(s), \mu^{N,i}_{\bm{Y}}(s)) \nonumber\\
&\qquad \quad  - V^N(s,Y_i(s), \mu^{N,i}_{\bm{Y}}(s)) + U(s,Y_i(s), \mu^{N,i}_{\bm{Y}}(s))| \nonumber\\
&\leq 2 \max_{\mu^N\in S_N^\varepsilon} |V^N(s,\mu^N) - U(s,\mu^N)|.
\label{13}
\end{align}
\begin{proof}[Proof of Theorem \ref{value}]
We choose a deterministic initial condition $\bm{Y}_{t}  \in \Sigma_N^\varepsilon$, at time $t\in[0,T)$.   As in the proof of Theorem 3 in \cite{cp}, we exploit the characterization, introduced in 
	\cite{cf}, of the $N$-player dynamics in terms of SDEs driven by Poisson random measures, and we apply Ito's formula to the squared difference between the functions $u_t^{N,i}$ and $v_t^{N,i}$, both computed in the optimal trajectories $(\bm{Y}_s)_{s\in[t,T]}$
	\footnote{We remark that in {\r \cite{cp}}, indeed, the controls (transition rates) are assumed to be bounded below away from zero. Nevertheless, this fact is not used to derive the analogous identity to \eqref{idem}. A proof of the convergence results with no lower bound on the controls can be found in Section 3.1 of \cite{tesi_alekos}, if the master equation possesses a classical solution.	
}	
	. Using equations \eqref{eqn:u_N} and \eqref{V^N}, we then find
\begin{align}
\mathbb{E}&[(u_t^{N,i} - v_t^{N,i})^2] + \sum_{j=0}^{N} \mathbb{E}\Bigg[\int_t^T \!\! \ \alpha^j(s,\bm{Y}_s)\Big(\Delta^j[u_s^{N,i} - v_s^{N,i}]\Big)^2{\r ds}\Bigg]\label{idem}\\ 
& =  -2\mathbb{E}\Bigg[\int_t^T\!\!\! (u_s^{N,i} - v_s^{N,i})\Bigg\{-r^N\!(s,\mu^{N,i}_{\bm{Y}}(s)) + H(\Delta^i u^{N,i}_s) - H(\Delta^i v^{N,i}_s) \nonumber\\
& + \sum_{j=0, j \neq i}^{N} (\alpha^j - \overline{\alpha}^j)\Delta^j u^{N,i} + \alpha^i (\Delta^i u^{N,i}_s - \Delta^i v^{N,i}_s)\Bigg\}{\r ds}\Bigg],  \nonumber 
\end{align}
where $\alpha^i$ is the Nash equilibrium played by player $i$, $\overline{\alpha}^i$ is the control induced by $U$ and all the functions are evaluated on the optimal trajectories, e.g. $v_s^{N,i} := v^{N,i}(s,\bm{Y}_s)$. We raise all the positive sum on the lhs and estimate the rhs using the Lipschitz properties of $H$, the bounds on $r^{N}$ and  $\Delta^j u^i$ given by Proposition \ref{propcp},  and the bound on $\alpha^j$ given by the fact that $Z^N(t,\mu) \leq 2$, to get, 
 for $N\geq \frac 2\varepsilon$,   
\begin{align*}
\mathbb{E}&[(u_t^{N,i} - v_t^{N,i})^2] \\
& \leq \frac{C}{ N}\mathbb{E}\Bigg[ \int_t^T |u^{N,i}_s - v^{N,i}_s | {\r ds} \Bigg] + C \mathbb{E}\Bigg[ \int_t^T |u^{N,i}_s - v^{N,i}_s | |\Delta^i u^{N,i}_s - \Delta^i v^{N,i}_s |{\r ds}\Bigg] \\
& + \frac{C}{N}\sum_{j=0, j\neq i}^{N} \mathbb{E}\Bigg[ \int_t^T |u^{N,i}_s - v^{N,i}_s | |\Delta^j u^{N,j}_s - \Delta^j v^{N,j}_s| {\r ds}\Bigg], 
\end{align*}
which can be further estimated via the convexity inequality $ab \leq \frac 12 a^2 + \frac 12 b^2$ yielding 
\begin{align*}
\mathbb{E}[(u_t^{N,i} - v_t^{N,i})^2]   
&\leq \frac{C}{ N^2} + C \mathbb{E}\Bigg[\int_t^T \!\!  \ \Big| u_s^{N,i} - v_s^{N,i}\Big|^2{\r ds}\Bigg] + C \mathbb{E}\Bigg[\int_t^T \!\!  \ \Big| \Delta^i u_s^{N,i} - \Delta^i v_s^{N,i}\Big|^2{\r ds}\Bigg]\\
& + \frac{C}{N}\sum_{j=0}^{N} \mathbb{E}\Bigg[\int_t^T|\Delta^j u^{N,j}_s - \Delta^j v^{N,j}_s|^2 {\r ds}\Bigg].
\end{align*}
Here $C$ denotes any constant which may depend on $\varepsilon$,  and is allowed to change from line to line. Since all the functions are evaluated on the optimal trajectories, we apply  
\eqref{12} and \eqref{13} to obtain
$$
|u^{N,i}(t,\bm{Y}_{t}) - v^{N,i}(t,\bm{Y}_{t})|^2 \leq \frac{C}{ N^2} + C \int_t^T \!\! 
 \max_{\mu\in S_N^\varepsilon} | U(s,\mu) - V^N(s,\mu)|^2 ds 
$$
for any deterministic initial condition $\bm{Y}_{t}\in \Sigma_N^\varepsilon$. Therefore \eqref{34} gives
\be
\max_{\mu\in S_N^\varepsilon} |U(t,\mu) - V^N(t,\mu)|^2 \leq \frac{C}{ N^2} +  C \int_t^T \!\! 
 \max_{\mu\in S_N^\varepsilon} |U(s,\mu) - V^N(s,\mu)|^2 ds 
\ee
and thus Gronwall's lemma applied to the quantity $\max_{\mu\in S_N^\varepsilon} |U(s,\mu) - V^N(s,\mu)|^2$ allows to conclude that
\be
 \max_{\mu\in S_N^\varepsilon} |U(t,\mu) - V^N(t,\mu)|^2 \leq \frac{C}{ N^2},  
\ee
which immediately implies \eqref{14}, but only if $N\geq \frac2\varepsilon$. Changing the value of $C=C_\varepsilon$, the thesis follows for any $N$.
\end{proof}

\subsection{Propagation of chaos} \label{SubChaos}

The  next result gives the propagation of chaos property for the optimal trajectories. Consider the initial datum (in$t=0$)  $\bm{\xi}$ i.i.d with $P(\xi_i =1)=\mu_0$ and $\E [\xi_i] = m_0 = 2\mu_0 - 1$, and denote by $\bm{Y}_t = ({\r Y_0(t),} Y_1(t),\dots,Y_N(t))$ the optimal trajectories of  the  {\r $N\!+\!1$}-player game, i.e.\ when agents play the Nash equilibrium given by \eqref{optimalnash}. Also, denote by  $\widetilde{\bm{X}}_t$ the i.i.d process in which players choose the local control $\widetilde{\alpha}(t,\pm 1):= [Z(t, m^*(t))]^{\mp}$, where $Z$ is the entropy solution to 
\eqref{conslaw} and $m^*$ is the unique mean field game solution induced by $Z$, if $m_0\neq0$ ($\mu_0\neq\frac12$), that is  the one which does not change sign (see Proposition \ref{mindotta}).  The propagation of chaos consists in proving the convergence of $\bm{Y}_t$ to the i.i.d process $\widetilde{\bm{X}}_t$.
\begin{thm}[Propagation of chaos]
\label{chaos}
If $\mu_0 \neq \frac12$ then, for any $N$ and $i=0,\dots,N$,
\begin{equation}
\mathbb{E}\left[\sup_{t\in[0,T]}|Y_i(t) - \widetilde{X}_i(t)|\right] \leq \frac{C_{\mu_0}}{\sqrt{N}},
\end{equation}
where $C_{\mu_0}$ does not depend on $N$, and $\lim_{\mu_0\rightarrow\frac12} C_{\mu_0}=\infty$.
\end{thm}
Denote by $X_i(t)$ the dynamics of the $i$-th player when choosing the control 
\begin{equation}
\label{eqn:mezz}
\bar{\alpha}^i(t,\bm{x}) = [\Delta^i U(t, x_i,\mu^{N,i}_{\bm{x}})]^-
\end{equation}
induced by the master equation.  We use $\bm{X}_t$ as an intermediate process for obtaining the propagation of chaos result. In fact, $\bm{X}_t$ can be treated as a mean field interacting system of particles (since the rate in \eqref{eqn:mezz} {\r depends on $N$ only through the empirical measure}), for which propagation of chaos results are more standard. Next result shows the proximity of the optimal dynamics to the intermediate process just introduced.

\begin{thm}
\label{XY}
If $\mu_0 \neq \frac12$ then, for any $N$ and $i=0,\dots,N$,
\begin{equation}
\mathbb{E}\left[\sup_{t\in[0,T]}|Y_i(t) - X_i(t)|\right] \leq \frac{C_{\mu_0}}{N},
\label{17}
\end{equation}
where $C_{\mu_0}$ does not depend on $N$, and $\lim_{\mu_0\rightarrow\frac12} C_{\mu_0}=+\infty$.
\end{thm}

\begin{proof}
Let $\mu_0= \frac12 +2 \varepsilon$ and consider the {\r event} $A$ where both $\bm{X}_t$ and $\bm{Y}_t$ belong to $\Sigma_N^\varepsilon$, for any time. Exploting the probabilistic representation of the dynamics in terms of Poisson random measures (see \cite{cf}), we have
\begin{align*}
\E &\left[\sup_{s\in[0,t]}|X_i(s) - Y_i(s)|\right]  \\
&\leq C\mathbb{E} \! \left[\int_{0}^t \left[\left|a^* (X_{i,s}, \Delta^i u^{N,i}(s,\bm{X}_s)) \! -\! a^*(Y_{i,s}, \Delta^i v^{N,i}(s,\bm{Y}_s))\right|+ \left|X_{i,s} \!-\! Y_{i,s}\right|\right]ds\right]\nonumber\\
&\leq C \E \left[\int_0^t \left[|X_i(s)-Y_i(s)| + |\Delta^i u^{N,i}(s,\bm{X}_s) - \Delta^i v^{N,i}(s,\bm{Y}_s)|\right] ds\right]\\
&\leq C \E \left[\int_0^t |X_i(s)-Y_i(s)|ds\right] + 
C \E\left[\mathbbm{1}_A \int_0^t |\Delta^i u^{N,i}(s,\bm{Y}_s) - \Delta^i v^{N,i}(s,\bm{Y}_s)|ds\right]\\
&+ C \E\left[\mathbbm{1}_A \int_0^t |\Delta^i u^{N,i}(s,\bm{X}_s) - \Delta^i u^{N,i}(s,\bm{Y}_s)|ds\right] + C P(A^c).
\end{align*}
and now  we apply \eqref{14} together with \eqref{13}, the Lipschitz continuity of $U$ in $\Sigma_N^\varepsilon$ and the exchangeability of the processes to get,
if $N\geq \frac 2\varepsilon$,
\begin{align}
&\E \left[\sup_{s\in[0,t]}|X_i(s) - Y_i(s)|\right] \leq \frac{C}{N} +C \int_0^t \E|X_i(s)-Y_i(s)|ds + 
P(A^c) \nonumber\\
& +  C \E\Bigg[\mathbbm{1}_A \int_0^t \left[|U(s, X_i(s), \mu^{N,i}_{\bm{X}}(s)) - U (s,X_i(s), \mu^{N,i}_{\bm{Y}}(s))| \right.\nonumber\\
&\left.\hspace{2 cm} + |U(s, -X_i(s), \mu^{N,i}_{\bm{X}}(s)) - U (s,-X_i(s), \mu^{N,i}_{\bm{Y}}(s))|\right]ds\Bigg] \nonumber\\
&\leq \frac{C}{N} +C \int_0^t \E|X_i(s)-Y_i(s)|ds + 
P(A^c) +  C \E \left[ \mathbbm{1}_A \int_0^t \frac 1N \sum_{j\neq i} |X_j(s)-Y_j(s)|ds\right]\nonumber\\
&\leq  \frac{C}{N} + C \int_0^t \E|X_i(s)-Y_i(s)|ds + P(A^c). \label{15}
\end{align}

We can bound the probability of $A^c$ by considering the process in which the  transition rates are equal to 0, for any time, i.e. the constant process equal to the initial condition $\bm{\xi}$. Thanks to the shape of the Nash equilibrium,  which prevents the dynamics from crossing the discontinuity, and of the control induced by the solution to the Master equation, we have 
\be
P(A^c)= P(\exists t : \mbox{ either } \bm{X}_t \mbox{ or } \bm{Y}_t\notin \Sigma_N^\varepsilon) \leq 
2 P(\bm{\xi}\notin \Sigma_N^\varepsilon). 
\ee
For the latter, we have 
\begin{align*}
P(\bm{\xi} \notin \Sigma_N^\varepsilon)& = P\left(\sum_{i=0}^N \xi_i \in \left(\frac{N}{2} - N\varepsilon, \frac{N}{2} + N\varepsilon+1 \right)\right)\nonumber\\
& \leq P\left(\sum_{i=0}^N \xi_i \leq \frac{N}{2} + N\varepsilon+1\right)
\leq P\left(\mu^N_{\bm{\xi}}\leq \frac12+ \varepsilon_N\right),
\end{align*}
denoting $\varepsilon_N := \frac{\frac{N}{2} + N\varepsilon + 1}{N+1} - \frac{1}{2}$. Observing that $(N+1)\mu^N_{\bm{\xi}} \sim \mathrm{Bin} (N+1, \frac12 + 2\varepsilon)$ {\r (recall $\mu_0 = \frac{1}{2} +2 \varepsilon$)}, we can further estimate, by standard Markov inequality,
\begin{align}
P(\bm{\xi} \notin \Sigma_N^\varepsilon)& \leq P\left(\left|\mu^N_{\bm{\xi}} - \frac 12 - 2\varepsilon\right| \geq 2\varepsilon - \varepsilon_N\right) 
\leq \frac{\mathrm{Var}\left[\mu^N_{\bm{\xi}}\right]}{(2\varepsilon - \varepsilon_N)^2} \nonumber\\
&= \frac{1}{N+1} \frac{\left(\frac12 +2\varepsilon\right) \left(\frac12 -2\varepsilon\right)}{\left(2\varepsilon - \frac{N}{N+1}\left(\frac12 +\varepsilon\right) - \frac{1}{N+1}+\frac12\right)^2}
\leq \frac {C}{N \varepsilon} \label{16}
\end{align}
if $N\geq \frac 2\varepsilon$, so that $2\varepsilon - \varepsilon_N \geq \frac \varepsilon4$.

Putting estimate \eqref{16} into \eqref{15}, and denoting $\varphi(t):= \E \left[\sup_{s\in[0,t]}|X_i(s) - Y_i(s)|\right]$, 
we obtain
\be
\varphi(t) \leq \frac{C}{N\varepsilon} + C \int_0^t \varphi(s)ds
\ee 
which, by Gronwall's lemma, gives \eqref{17}, but only if $N\geq \frac 2\varepsilon$. By changing the value of $C=C_\varepsilon$, the claim follows for any $N$.
\end{proof}

We are now in the position to prove Theorem \ref{chaos}. Thanks to \eqref{17}, it is enough to show that
\be
\mathbb{E}\left[\sup_{t\in[0,T]}|X_i(t) - \widetilde{X}_i(t)|\right] \leq \frac{C_{\mu_0}}{\sqrt{N}},
\ee
Recall that the $\widetilde{X}_i$'s are i.i.d and  $\text{Law}(\widetilde{X}_i(t)) = m^*(t)$; also, set $m=m^*$ and $\mu=\frac{m+1}{2}$. {\r Moreover, we know that $(N+1)\mu^{N}_{\widetilde{\bm{X}}}(t) \sim \mathrm{Bin}(N+1, \mu(t))$.} The rate of convergence follows from the estimate 
\be
\E \left| \mu^{N}_{\widetilde{\bm{X}}}(t) - \mu(t) \right|\leq \frac{C}{\sqrt{N}},
\ee
for any time, {\r by Cauchy-Schwarz inequality}.

\begin{proof}[Proof of Theorem \ref{chaos}]

Let $\mu_0= \frac12 +2 \varepsilon$ and consider the {\r event} $A$ where both $\bm{X}_t$ and $\widetilde{\bm{X}}_t$ belong to $\Sigma_N^\varepsilon$, for any time. Arguing as in the proof of Theorem \ref{XY}, we obtain
\begin{align*}
\E &\left[\sup_{s\in[0,t]}|X_i(s) - \widetilde{X}_i(s)|\right] \leq C \int_0^t \E|X_i(s)-\widetilde{X}_i(s)|ds + 
P(A^c) \\
& \hspace{2 cm} +  C \E\Bigg[\mathbbm{1}_A \int_0^t |U(s, X_i(s), \mu^{N,i}_{\bm{X}}(s)) - U(s,X_i(s), \mu^{N,i}_{\bm{\widetilde{X}}}(s))| \\
&  \hspace{3.5 cm} + |U(s, -X_i(s),  \mu^{N,i}_{\bm{\widetilde{X}}}(s)) - U (s,-X_i(s),\mu(s))|ds\Bigg] \\
&\hspace{ 2cm }\leq C \int_0^t \E|X_i(s)-\widetilde{X}_i(s)|ds + P(A^c)  \\
& \hspace{3.5 cm}  +  C \E \left[ \mathbbm{1}_A \int_0^t \frac 1N \sum_{j\neq i} |X_j(s)-\widetilde{X}_j(s)|ds\right] + C\sup_{t\in [0,T]} \E\left| \mu^{N}_{\widetilde{\bm{X}}}(t) - \mu(t) \right|\\
& \hspace{2 cm} \leq  \frac{C}{\sqrt{N}} + C \int_0^t \E|X_i(s)-\widetilde{X}_i(s)|ds + P(A^c). 
\end{align*}
We can bound the probability of $A^c$ as before and thus Gronwall's Lemma allows to conclude.
\end{proof}

\subsection{Potential mean field game} \label{SubPotential}

 We give here another characterization of the solutions to the MFG system \eqref{mfg}. For a more detailed introduction on potential mean field games in the finite state space see \cite{tesi_alekos}, Section 1.4.1. We show that system \eqref{mfg} can be viewed as the {\r necessary conditions for optimality}, given by the Pontryagin maximum  principle, of a \emph{deterministic} optimal control problem in $\mathbb{R}^2$. We show that the $N$-player game,
in the limit as $N \rightarrow +\infty$,
 selects exactly the {\r global minimizer} of this problem when it is unique, i.e. when $m_0\neq0$. 

The notation {\r is} slightly different in this section. Consider the controlled dynamics, representing the KFP equation,
\be
\begin{cases}
\dot{m}_1 = m_{-1} \alpha_{-1} -m_1\alpha_1\\
\dot{m}_{-1} = m_1 \alpha_1 -m_{-1}\alpha_{-1}\\
m(0)=m_0.
\end{cases}
\label{kol}
\ee
The state variable is $m(t)=(m_1(t),m_{-1}(t))$. Note that, in the previous notation, we had $m_1=\mu$ and $m=m_1-m_{-1}$. Here the control is $\alpha(t) =(\alpha_1(t),\alpha_{-1}(t))$, deterministic and open-loop, taking  values in
$$A=\left\{(a_1,a_{-1}) : a_1,a_{-1} \geq 0\right\}.$$
Clearly, if $m_0 = (m_{0,1}, m_{0,-1})$ belongs to the simplex
$$P(\left\{1,-1\right\}):= \left\{ (m_1,m_{-1}) : m_1+m_{-1} =1, m_1,m_{-1} \geq0\right\},$$
then, for any choice of the control $\alpha$, the dynamics remains in $P(\left\{1,-1\right\})$ for any time.

The cost to be minimized is 
\be
\mathcal{J}(\alpha)= \int_0^T \left(m_1(t) \frac{\alpha_1(t)^2}{2} + m_{-1}(t)\frac{\alpha_{-1}(t)^2}{2}\right)dt + \mathcal{G}(m(T)),
\label{costo}
\ee
where $\mathcal{G}(m_1, m_{-1}):= - \frac{(m_1-m_{-1})^2}{2}$ is such that
\begin{align*} 
\frac{\partial}{\partial m_1} \mathcal{G}(m)&= -(m_1-m_{-1})=: G(1,m) \\ 
\frac{\partial}{\partial m_{-1}} \mathcal{G}(m)&= m_1-m_{-1}=: G(-1,m),
\end{align*}
whereas $G(x,m)= -x (m_1-m_{-1})$, for $x=\pm 1$, is the terminal cost. This structure is called \emph{potential} Mean Field Game, since we have $\nabla \mathcal{G}(m) = G(\cdot,m)$. 

The Hamiltonian of this problem is 
\begin{align*}
\mathcal{H}(m,u)&= \sup_{a\in A} \left\{-b(m,a)\cdot u - m_1 \frac{a_1^2}{2} - m_{-1}\frac{a_{-1}^2}{2}\right\}\\ 
&=m_1 \frac{[(u_{-1}-u_1)^-]^2}{2} + m_{-1}\frac{[(u_{1}-u_{-1})^-]^2}{2},
\end{align*}
where $b_{x}(m,a) = m_{-x}a_{-x} - m_x a_x$, for $x=\pm 1$, is the vector field in \eqref{kol}, 
and the argmax of the Hamiltonian is 
\begin{align*}
a^*_1(u)&= (u_{-1}-u_1)^- ,\\
a^*_{-1}(u)&= (u_{1}-u_{-1})^-.
\end{align*}
Thus, the HJB equation of the control problem reads 
\be
\begin{cases}
&-\frac{\partial \mathcal{U}}{\partial t} + \mathcal{H}(m, \nabla_m \mathcal{U})=0 \qquad t\in [0,T), m\in\mathcal{P}(\left\{1,-1\right\})\\
&\mathcal{U}(T,m)=\mathcal{G}(m),
\end{cases}
\label{value2}
\ee
and its characteristics curves are given by the MFG system
\be
\begin{cases}
-\dot{u}_1 + \frac{[(u_{-1}-u_1)^-]^2}{2}=0\\
-\dot{u}_{-1} + \frac{[(u_{1}-u_{-1})^-]^2}{2}=0\\
\dot{m}_1 = m_{-1} a_{-1}^*(u) -m_1 a_1^*(u)\\
\dot{m}_{-1} = m_1 a_1^*(u) -m_{-1}a_{-1}^*(u)\\
u_{\pm 1}(T)=G(\pm 1,m(T)), \quad m(0)=m_0.
\end{cases}
\label{pon}
\ee

{\r \begin{lem}
\begin{enumerate}
	\item There exists an optimum of the control problem \eqref{kol}-\eqref{costo};
	\item The MFG system \eqref{pon} represents the necessary conditions for optimality, given by the Pontryagin maximum principle.
\end{enumerate}
\end{lem}}

\begin{proof}
The first claim follows from Theorem 5.2.1 p. 94 in \cite{brepi}, which can be applied since the dynamics is linear in $\alpha$ and the running cost is convex in $\alpha$.
Conclusion (2) is standard.
\end{proof}

We know that, if $T$ is large enough, there are three solutions to the MFG system. The control problem \eqref{kol}-\eqref{costo} has a minimum, so we wonder which of these solutions is indeed a minimizer.

First, we need to investigate some property of the roots of \eqref{cons}. Let  $T>T(m_0)$ be fixed. Let $M_1(m_0)<M_2(m_0)<M_3(m_0)$ be the three solutions to \eqref{cons}. If $m_0=0$ denote $M_-=M_1(0)<0$, $M_+=M_3(0)>0$; we have $M_2(0)=0$ and $M_+=M_-$. If $m_0>0$ then, by Proposition \ref{prop1}, $M_3(m_0)>0$ and $M_1(m_0),M_2(m_0)<0$; if $m_0<0$ then $M_3(m_0)<0$ and $M_1(m_0),M_2(m_0)>0$.

\begin{lem}
Let $m_0>0$ and $T>T(m_0)$ be fixed. Then
\begin{enumerate}
	\item The function $[0,m_0]\ni m\mapsto M_3(m) \in [0,1]$ is increasing, $M_2(m)$ is decreasing and $M_1(m)$ is increasing. In particular
	for any $m\in [0,m_0]$
	\be
	M_3(m)>M_+=|M_-|>|M_1(m)|>|M_2(m)|> M_2(0)=0
	\label{cat}
	\ee
	\item We have $M_1(m)<-\frac{2T-1}{3T}<M_2(m)<0$ and for any $m\in [0,m_0]$
	\be
    \left| M_2(m) +\frac{2T-1}{3T}\right| > \left| M_1(m) +\frac{2T-1}{3T}\right|.
	\label{maggio}
	\ee
\end{enumerate}
The case $m_0<0$ is symmetric. 
\label{lemmon}
\end{lem}

\begin{proof}
Claim (1) derives from the proof of Proposition \ref{prop1}. For claim (2), $M_1(m)$ and $M_2(m)$ are the two negative roots of
$f(M)= T^2 M^3 - T(2-T)M^2 + (1-2T)M - m =0.$
The roots of  $f'(M)$ are $q:=-\frac{2T-1}{3T}$ and $\frac 1T$. Hence $M_1<q<M_2<0$, $f(q)>0$ and we have, by Taylor's formula (which here is actually a change of variable),
\begin{align*}
f(q+\varepsilon) &=f(q)+f'(q)\varepsilon + \frac{f''(q)}{2}\varepsilon^2 + \frac{f'''(q)}{6}\varepsilon^3 =
f(q)+ \frac{f''(q)}{2}\varepsilon^2 +T^2\varepsilon^3\\
f(q-\varepsilon) &=f(q)-f'(q)\varepsilon + \frac{f''(q)}{2}\varepsilon^2 - \frac{f'''(q)}{6}\varepsilon^3 =
f(q)+ \frac{f''(q)}{2}\varepsilon^2 -T^2\varepsilon^3
\end{align*}
for any $\varepsilon>0$. Thus $f(q+\varepsilon)-f(q-\varepsilon)= 2T^2\varepsilon^3>0$ for any $\varepsilon>0$, which implies \eqref{maggio}.
\end{proof}

For $i=1,2,3$, denote by $m_i, z_i, \alpha_i,m_i, u_i$ the solution to the MFG system corresponding to $M_i$.

\begin{thm}
\label{selects}
Let $m_0>0$ and $T>T(m_0)$ be fixed. Then for any $m\in[0,m_0]$ and $i=1,2,3$ we have  $\mathcal{J}(\alpha_i)=\varphi(M_i(m))$, where
$\varphi:[-1,1]\rightarrow [-1,1]$, 
\be
\varphi(M) := M^2 \left(T-\frac12 -T|M|\right).
\label{phi}
\ee
Moreover, for any $m \in \ (0,m_0]$,
\begin{align}
\varphi(M_+)&=\varphi(M_-)<\varphi(0)=0,\label{dis1}\\
\varphi(M_3(m))&<\varphi(M_+)<\varphi(M_1(m)),\label{dis2}\\
\varphi(M_1(m))&<\varphi(M_2(m))>0,\label{dis3}
\end{align}
meaning that $\alpha_+$ and $\alpha_-$ are both optimal if $m=0$ and $\alpha\equiv 0$ is not, while $\alpha_3$ is the unique {\r minimizer} 
if $m>0$, with 
\be
 \mathcal{J}(\alpha_3)< \mathcal{J}(\alpha_1)< \mathcal{J}(\alpha_2).
\ee
\end{thm}

\begin{proof}
The first claim and \eqref{phi} follow directly from \eqref{costo} and \eqref{solkol}. 

We  continue by proving \eqref{dis2}. The roots of $\varphi'$ are 0 and $\pm q$, with $q:= -\frac{2T-1}{3T}$. The function $\varphi$ is then increasing if either $M<q$ or $0<M<-q$.  Thus  \eqref{dis2} follows from \eqref{cat} and the fact that $\varphi(M_+)=\varphi(M_-)$, as 
$\varphi(M)$ only depends on  $|M|$. 

Next, we show that $\varphi(M_+)< 0=\varphi(0)$. Since $M_+$ solves $T^2 M^2 + T(2-T)M +1-2T=0$, we obtain, for $M=M_+$,
$$\varphi(M)= \frac{M^2}{2} (2T-1-2TM)= \frac{M^2}{2} (T^2 M^2 - T^2 M)= \frac{T^2 M^3}{2}(M-1)<0$$
because $M_+<1$. 

To prove \eqref{dis3}, we first note that we have just showed that it holds in $m=0$: 
$\varphi(M_1(0))=\varphi(M-)=\varphi(M_+)<0=\varphi(0)=\varphi(M_2(0))$. We also know that $\varphi(M_1(m))>\varphi(M_1(0))$ and 
$\varphi(M_2(m))>\varphi(M_2(0))$, thanks to the monotonicity behavior of $\varphi$ and Lemma \ref{lemmon}. Hence suppose by contradiction that there exists $m\in ]0,m_0]$ such that $\varphi(M_1(m))=\varphi(M_2(m)) =c$, for some $c>0$. This implies that both $M_1(m)$ and $M_2(m)$ are negative roots of $\varphi(M) -c$. Thus they are also negative roots of
$$\psi(M):= T\varphi(M)- Tc - f(M)= \frac32 T M^2 -(1-2T)M + m - Tc =0$$
and $\psi'(q)=0$, where $q= -\frac{2T-1}{3T}$ as above. Since $\psi$ has degree 2, it follows that $|M_2(m)-q|=|M_1(m)-q|$, but this contradicts \eqref{maggio}. Therefore there is no $m$ for which $\varphi(M_1(m))=\varphi(M_2(m))$, and then if \eqref{dis3} holds for $m=0$ (which is \eqref{dis1}) then it is true for any $m\in[0,m_0]$.
\end{proof}
 Note that 
the results in this section imply that the $N$-player game selects, in the limit as $N \ra +\infty$,
the global minimizer of the control problem \eqref{costo}, when it is unique. Moreover, the sequence of the $N$-player value functions $V^N$ converges to the derivative of the value function of such control problem, as the latter is constructed by using the same characteristic curves used for constructing the solution \eqref{zstar} to the master equation. We remark that the value function of the control problem \eqref{costo} can also be characterized as the unique viscosity solution to \eqref{value2}.

\section{Conclusions} \label{Conclusions}

Let us summarize the main results we have obtained for this two state model with anti-monotonous terminal cost:
\begin{enumerate}
	\item the mean field game possesses exactly 3 solutions, if $T>2$ (Proposition \ref{prop1});
	\item the $N$-player value functions converge to the entropy solution to the master equation (Theorem \ref{value});
	\item the $N$-player optimal trajectories converge to one mean field game solution, if $m_0\neq0$ (Theorem \ref{chaos});
	\item viewing the mean field game system as the necessary conditions for optimality of a deterministic control problem, the $N$-player game selects the {\r global minimizer} of this problem, when it is unique, i.e. $m_0\neq0$ (Theorem \ref{selects}).
\end{enumerate}

We remark that in the convergence proof we did not make use of the characterization of the right solution to the master equation as the entropy admissible one; the key point is to show that the $N$-player optimal trajectories do not cross the discontinuity. Neither did we use the potential structure of the problem: these are properties which might allow to  extend the convergence results to more general models.

Observe that solutions of the MFG system, whether selected by the limit of $N$-player Nash equilibria or not, always yield approximate Nash equilibria in decentralized symmetric feedback strategies; see, for instance, \cite{basna} and \cite{cf} in the finite state setting.

What is left to prove for this model is a propagation of chaos result when $m_0=0$. Let $m_+$, resp. $m_-$, be the mean field game solution always positive, resp.\ always negative. What is evident from the simulations is that the $N$-player optimal trajectories admit a limit which is not deterministic: it is supported {\r in} $m_+$ and $m_-$ with probability $1/2$. We also observe  that $m_+$ and $m_-$ are both {\r minimizers} of the deterministic optimal control problem related to the potential structure.  An analogous result is rigorously obtained in \cite{delaruetchuendom} in the diffusion setting, where the focus is on starting the dynamics at the discontinuity of the unique entropy solution to the master equation.



\begin{thebibliography}{9}

\bibitem{bardifischer}
M.~Bardi and M.~Fischer.
\newblock On non-uniqueness and uniqueness of solutions in finite-horizon Mean Field Games
\newblock \emph{\mbox{ESAIM: COCV}}, pre-published on-line, \texttt{https://doi.org/10.1051/cocv/2018026}, 2018.

\bibitem{basna}
R.~Basna, A.~Hilbert, and V.~N. Kolokoltsov.
\newblock An epsilon-Nash equilibrium for non-linear Markov games of mean-field-type on finite spaces.
\newblock \emph{Comm. Stoch. Anal.}, 8\penalty0 (4):\penalty0 449--468, 2014.

\bibitem{bayraktarcohen}
E.~Bayraktar and A.~Cohen.
\newblock Analysis of a finite state many player game using its master equation.
\newblock  \emph{SIAM Journal on Control and Optimization}, 56(5): 3538-3568, 2018.

\bibitem{brepi}
A.~Bressan, B.~Piccoli.
\newblock \emph{Introduction to the Mathematical Theory of Control}, volume~2 of \emph{\mbox{AIMS} Series on Applied Mathematics}.
\newblock American Institute of Mathematical Sciences, Springfield \mbox{(MO)}, 2007.

\bibitem{cardaliaguet13}
P.~Cardaliaguet.
\newblock Notes on mean field games.
\newblock Technical report, Universit{\'e} de Paris - Dauphine, September 2013.

\bibitem{card}
P.~Cardaliaguet, J.~Graber, A.~Porretta, and D.~Tonon. 
\newblock Second order mean field games with degenerate diffusion and local coupling.
\newblock \emph{NoDEA}, 22(5): 1287–1317, 2015.

\bibitem{cardaliaguetetal15}
P.~Cardaliaguet, F.~Delarue, J.-M.\ Lasry, and P.-L.\ Lions.
\newblock The master equation and the convergence problem in mean field games.
\newblock \texttt{arXiv:1509.02505 [math.AP]}, September 2015.

\bibitem{carmonadelarue}
R.~Carmona and F.~Delarue.
\newblock \emph{Probabilistic Theory of Mean Field Games with Applications}, volumes 83 and 84 of \emph{Probability Theory and Stochastic Modelling}.
\newblock Springer, Cham, 2018.

\bibitem{tesi_alekos}
A.~Cecchin.
\newblock Finite State N-player and Mean Field Games.
\newblock PhD Thesis, Department of Mathematics, University of Padua, 2018.

\bibitem{cf}
A.~Cecchin and M.~Fischer.
\newblock Probabilistic approach to finite state mean field games.
\newblock \emph{Appl. Math. and Optim.}, pre-published on-line, \texttt{https://doi.org/10.1007/s00245-018-9488-7}, 2018.
{\r
\bibitem{cp}
A.~Cecchin and G.~Pelino.
\newblock Convergence, fluctuations and large deviations for finite state mean field games via the master equation. 
\newblock \emph{Stochastic Processes and their Applications}, pre-published online, \texttt{doi.org/10.1016/j.spa.2018.12.002}, December 2018.
}
\bibitem{dafermos}
C. M.~Dafermos.
\newblock \emph{Hyperbolic Conservation Laws in Continuum Physics}, volume 325 of \emph{Grundlehren der mathematischen Wissenschaften}.
\newblock Springer, Berlin Heidelberg, 4th ed., 2016.

 \bibitem{paolo}
P.~Dai Pra, E~Sartori, and M~Tolotti.
\newblock  Climb on the Bandwagon: Consensus and periodicity in a lifetime utility model with strategic interactions.
\newblock \texttt{arXiv:1804.07469 [math.OC]}, April 2018.

\bibitem{delarueetala}
F.~Delarue, D.~Lacker, and K.~Ramanan.
\newblock From the master equation to mean field game limit theory: a central limit theorem.
\newblock \texttt{arXiv:1804.08542 [math.PR]}, April 2018.

\bibitem{delarueetalb}
F.~Delarue, D.~Lacker, and K.~Ramanan.
\newblock From the master equation to mean field game limit theory: large deviations and concentration of measure.
\newblock \texttt{arXiv:1804.08550 [math.PR]}, April 2018.

\bibitem{delaruetchuendom}
F.~Delarue and R.\ Foguen Tchuendom.
\newblock Selection of equilibria in a linear quadratic mean field game.
\newblock \texttt{arXiv:1808.09137 [math.PR]}, August 2018.

{\r{\bibitem{doncel}
J.~Doncel, N.~Gast, and B.~Gaujal. 
\newblock Mean-field games with explicit interactions. 
\newblock Preprint hal-01277098, February 2016.}
}
\bibitem{fischer17}
M.~Fischer.
\newblock On the connection between symmetric {$N$}-player games and mean field games.
\newblock \emph{Ann. Appl. Probab.}, 27\penalty0 (2):\penalty0 757--810, 2017.

\bibitem{gomes}
D.~Gomes, J.~Mohr, and R.R.\ Souza.
\newblock Continuous time finite state mean field games.
\emph{Appl. Math. Optim.}, 68(1):\penalty0 99--143, 2013.

\bibitem{gomestwoa}
D.~Gomes, R.M.\ Velho, M.-T.\ Wolfram.
\newblock Dual two-state mean field games.
\newblock 53rd \mbox{IEEE} Conference on Decision and Control. Los Angeles (\mbox{CA}), December 15--17, 2014.

\bibitem{gomestwob}
D.~Gomes, R.M.\ Velho, M.-T.\ Wolfram.
\newblock Socio-economic applications of finite state mean field games.
\newblock \emph{Phil. Trans. R. Soc. A}, 372:\penalty0 20130405, 2014.

\bibitem{huangetal06}
M.~Huang, R.P. Malham{\'e}, and P.E. Caines.
\newblock Large population stochastic dynamic games: {C}losed-loop {McKean-Vlasov} systems and the {N}ash certainty equivalence principle.
\newblock \emph{Commun. Inf. Syst.}, 6\penalty0 (3):\penalty0 221--252, 2006.

\bibitem{lacker16}
D.~Lacker.
\newblock A general characterization of the mean field limit for stochastic differential games.
\newblock \emph{Probab. Theory Related Fields}, 165\penalty0 (3):\penalty0 581--648, 2016.

\bibitem{lacker}
D.~Lacker.
\newblock On the convergence of closed-loop Nash equilibria to the mean field game limit.
\newblock \texttt{arXiv:1808.02745 [math.PR]}, August 2018.

\bibitem{lasrylions07}
J.-M. Lasry and P.-L. Lions.
\newblock Mean field games.
\newblock \emph{Japan. J. Math.}, 2\penalty0 (1):\penalty0 229--260, 2007.

\bibitem{nutzetalii}
M.~Nutz, J.~San~Martin, and X.~Tan.
\newblock Convergence to the mean field game limit: a case study.
\newblock \texttt{arXiv:1806.00817 [math.OC]}, June 2018.

\end{thebibliography}
\end{document}